\documentclass[12pt,reqno,a4paper]{amsart}
\usepackage{amsmath,amssymb,amsthm,amsaddr}
\usepackage{enumitem}
\usepackage[margin=2.5cm,top=3.5cm,bottom=3.5cm,footskip=1cm,headsep=1cm]{geometry}
\usepackage[utf8]{inputenc}
\usepackage{graphicx}

\author{H. Egger \and T. Kugler}
\address{Department of Mathematics, TU Darmstadt, Germany}
\email{egger@mathematik.tu-darmstadt.de}
\email{kugler@mathematik.tu-darmstadt.de}

\title[Exponential stability of Galerkin methods for damped wave systems]{Uniform exponential stability of Galerkin approximations for damped wave systems}

\newtheorem{lemma}{Lemma}[section]
\newtheorem{problem}[lemma]{Problem}
\newtheorem{theorem}[lemma]{Theorem}

\theoremstyle{definition}
\newtheorem{remark}[lemma]{Remark}
\newtheorem*{example*}{Example}

\def\dt{\partial_t}
\def\dx{\partial_x}
\def\dtt{\partial_{tt}}
\def\dxx{\partial_{xx}}

\def\u{u}

\def\RR{\mathbb{R}}

\def\eps{\varepsilon}

\def\E{E}

\def\k{k}
\def\Ek{\E^\k}
\def\dtk{\dt^\k}

\numberwithin{equation}{section}
\numberwithin{table}{section}
\numberwithin{figure}{section}

\begin{document}

\begin{abstract} 
We consider the numerical approximation of linear damped wave systems by Galerkin approximations in space and appropriate time-stepping schemes.
Based on a dissipation estimate for a modified energy, we prove exponential decay of the physical energy on the continuous level provided that the damping is effective everywhere in the domain. 
The methods of proof allow us to analyze also a class of Galerkin approximations based on a mixed variational formulation of the problem. Uniform exponential stability can be guaranteed for these approximations under a general compatibility condition on the discretization spaces. 
As a particular example, we discuss the discretization by mixed finite element methods for which we obtain convergence and uniform error estimates under minimal regularity assumptions.
We also prove unconditional and uniform exponential stability for the time discretization by certain one-step methods.
The validity of the theoretical results as well as the necessity of some of the conditions required for our analysis are demonstrated in numerical tests.
\end{abstract}

\maketitle

\begin{quote}
\end{quote}

\begin{quote}
\noindent
{\small {\bf AMS-classification (2000):}
35L05, 35L50, 65L20, 65M60}
\end{quote}

\section{Introduction} \label{sec:intro}

%
%
The propagation of sound waves in gas pipelines or of shock waves in water pipes, known as the water hammer, can be described by hyperbolic systems of the form \cite{BrouwerGasserHerty11,Guinot08}
\begin{align}
\dt u + \dx p + a u &= 0 \label{eq:sys1}\\
\dt p + \dx u &= 0. \label{eq:sys2}
\end{align}
In this context $p$ denotes the pressure, $\u$ is the velocity, and  
$a$ is the damping parameter accounting for friction at the pipe walls.
We will assume here that $a=a(x)$ is uniformly positive.
Variants of the above system also describe the damped vibration of a string or the heat transfer at small time and length scales \cite{Cattaneo48}. 
%
%
A suitable combination of the two differential equations leads to the second order form
\begin{align}
\dtt u +a \dt u - \dxx u = 0, \label{eq:sof}
\end{align}
known as the \emph{damped wave} or \emph{telegraphers equation} \cite{Evans98};
we will call the first order system \eqref{eq:sys1}--\eqref{eq:sys2} \emph{damped wave system}, accordingly.
Dissipative hyperbolic systems of similar structure describe rather general wave phenomena, 
in acoustics, linear elasticity, electromagnetics, heat transfer, or particle transport.
We consider here only a one-dimensional model problem in detail,
but most of our results can be generalized without much difficulty to more general and multi-dimensional 
problems of similar structure.

%
%
Due to many applications the damped wave equation has attracted significant interest in the literature. To put our work into perspective, let us briefly recall some of the main results.
%
%
When modeling the vibration of a string, the function $u$ denotes the displacement  and $\dt p=-\dx u$ the internal stresses. The physical energy of the system, consisting of a kinetic and a potential component, is then given by 
\begin{align*}
\E^1(t) = \frac{1}{2} \big( \|\dt u(t)\|^2 + \|\dt p(t)\|^2 \big).
\end{align*}
%
Replacing $\dt p$ by $\dx u$ yields the more common form $\frac{1}{2} \big( \|\dt u(t)\|^2 + \|\dx u(t)\|^2 \big)$ of the energy for the one-dimensional wave equation. 
If the boundary conditions are chosen such that no energy can enter or leave the domain via the boundary, then 
\begin{align*}
\frac{d}{dt} \E^1(t) = -\int_{dom} a \; |\dt u(t)|^2 dx \le 0,
\end{align*}
which illustrates that kinetic energy is dissipated efficiently 
by the damping mechanism.
%
%
It is well-known that as a consequence the total energy $\E^1(t)$ decreases exponentially, i.e., 
\begin{align*}
\E^1(t) \le C e^{-\alpha t} \E^1(0) 
\end{align*}
for some constants $C$ and $\alpha>0$, provided that the damping is effective at 
least on a sub-domain of positive measure \cite{CoxZuazua94,RauchTaylor74,Zuazua88}. 
A similar result holds if the damping only takes place at the boundary \cite{Chen79,Lagnese83}.
These decay estimates for the energy imply the exponential stability of the system which is of relevance from a practical and a theoretical point of view, e.g. for the control of damped wave systems \cite{Chen79,Zuazua05}. 

%
%
A first contribution of this paper is to show that the same decay estimates hold for the problem under investigation and for all energies of the form 
$$
\Ek(t) = \frac{1}{2} \big( \|\dtk u(t)\|^2 + \|\dtk p(t)\|^2 \big), \qquad k \ge 0
$$
under the assumption that the damping parameter $a(x)$ is bounded and uniformly positive.
In particular, our results cover the decay of $\E^0(t)=\frac{1}{2}( \|u(t)\|^2 + \|p(t)\|^2)$, 
which is the physical energy for the acoustic wave propagation or the water hammer problem.
This result is derived by carefully adopting arguments of \cite{BabinVishik83,Zuazua88} to the problem under consideration, in particular addressing the different boundary conditions and the case $k=0$. 
The assumptions on the damping parameter allow us to prove the energy decay under minimal regularity requirements on the damping parameter and the solution, and to obtain explicit estimates for the damping rate $\alpha$ depending only on the bounds for the damping parameter. 
This is important for the asymptotic analysis of damped wave phenomena \cite{LopezGomez97} and for the characterization of the parabolic limit problem.

%
%
Apart from the analysis, also the numerical approximation of damped wave phenomena has attracted significant interest in the literature, in particular for problems where the damping is not effective everywhere in the domain; see e.g. \cite{BanksItoWang91,ErvedozaZuazua09,Fabiano01,GlowinskiKintonWheeler89,InfanteZuazua99,TebouZuazua03,Zuazua05} and the references therein.
Mixed finite element schemes have been proven to be particularly well-suited for a systematic approximation \cite{BanksItoWang91,GlowinskiKintonWheeler89,Joly03,Zuazua05}. 
Error estimates for some fully discrete schemes for the damped wave equation have been obtained in \cite{GaoChi07,GroteMitkova13,Karaa11,RinconCopetti13}.
Let us also refer to  \cite{Baker76,DouglasDupontWheeler78,Dupont73,Geveci88,Joly03,Tebou98} for basic results on the analysis of numerical methods for wave propagation problems.
Our research contributes to this field by proposing and analyzing fully discrete approximation schemes 
that preserve the exponential stability uniformly with respect to the discretization parameters.

%
%
For the discretization in space, we consider here Galerkin approximations for a variational formulation of the first order system \eqref{eq:sys1}--\eqref{eq:sys2} in the spirit of \cite{Joly03}.
A simple compatibility condition for the approximation spaces for velocity and pressure allows us to establish well-posedness of a general class of Galerkin schemes and to prove the uniform exponential decay of the energies $E^k(t)$, $k \ge 0$ also on the semi-discrete level.
As a particular example for an appropriate Galerkin scheme, we discuss in some detail 
the discretization by mixed finite elements and we provide explicit error estimates for the resulting semi-discrete methods.
%
%
We then consider the time discretization by certain one-step methods and 
establish the unconditional and uniform exponential stability for the fully discrete schemes.
As a by-product of our stability analysis, we obtain error estimates that 
hold uniformly in time and with respect to the spatial and temporal mesh size.

\medskip 

%
%
The rest of the paper is organized as follows:
In Section~\ref{sec:prelim}, we give a complete definition of the problem under investigation and recall some basic results about the well-posedness of the problem.
In Section~\ref{sec:energy}, we derive the energy estimates and prove the exponential decay
to equilibrium under minimal regularity assumptions. 
In Section~\ref{sec:var}, we then provide a characterization of classical solutions via variational principles,
which are the starting point for the numerical approximation.
Section~\ref{sec:semi} is concerned with a general class of Galerkin discretizations in space, 
for which we provide uniform stability and error estimates. The discretization by mixed finite elements 
is discussed as a particular example in Section~\ref{sec:mixed}.
In Section~\ref{sec:time}, we then investigate the time discretization by a family of one-step methods and we prove unconditional and uniform exponential stability and error estimates for the fully discrete schemes.
Some numerical tests are presented in Section~\ref{sec:num} for illustration of the 
theoretical results.
We close with a short summary and briefly discuss the possibility for generalizations to multi-dimensional problems and other applications.

\section{Preliminaries} \label{sec:prelim}

For simplicity, we assume that the domain under consideration is the unit interval.\linebreak 
By $L^p(0,1)$ and $H^1(0,1)$ we denote the usual Lebesgue and Sobolev spaces.
The functions in $H_0^1(0,1)$ additionally vanish at the boundary. 
We denote by $(f,g)=\int_0^1 f g dx$ the scalar product 
on $L^2(0,1)$, and with $\|f\| = \|f\|_{L^2(0,1)}$ and $\|f\|_{1} = \|f\|_{H^1(0,1)}$ 
the norms of the spaces $L^2(0,1)$ and $H^1(0,1)$.
By $C^l([0,T];X)$ and $H^k(0,T;X)$ we denote the spaces of functions $f:[0,T] \to X$ 
with values in some Banach space $X$ having the appropriate smoothness and integrability 
with respect to time; see e.g. \cite{Evans98} for details.

The focus of our considerations lies on the linear hyperbolic system
\begin{align}
\dt u + \dx p + a u &= 0  \qquad \text{in } (0,1) \times (0,T)  \label{eq:instat1} \\
\dt p + \dx u       &= 0  \qquad \text{in } (0,1) \times (0,T)  \label{eq:instat2}
\end{align}
with homogeneous Dirichlet conditions for the pressure
\begin{align}
p&=0 \qquad   \text{on } \{0,1\} \times (0,T). \label{eq:instat3}
\end{align}
More general boundary conditions or inhomogeneous right hand sides 
could be taken into account without much difficulty. 
The initial values shall be prescribed by
\begin{align} \label{eq:instat4}
u(0)=u_0 \quad \text{and} \quad  p(0)=p_0 \qquad \text{on } (0,1). 
\end{align}
The well-posedness of this initial boundary value problem can  be deduced 
with standard arguments. For later reference, let us state the most basic results.
\begin{lemma} \label{lem:instat}
Let $a \in L^\infty(0,1)$ and $T>0$.
Then for any $u_0,p_0 \in L^2(0,1)$, the problem \eqref{eq:instat1}--\eqref{eq:instat4} has a unique mild solution 
$(u,p) \in C^0([0,T];L^2(0,1) \times L^2(0,1))$ and
$$
\max_{0 \le t \le T} \|u(t)\|^2 + \|p(t)\|^2 \le C \big( \|u_0\|^2 + \|p_0\|^2 \big).
$$
If in addition $u_0 \in H^1(0,1)$ and $p_0 \in H_0^1(0,1)$, then 
$(u,p)$ is a classical solution, 
in particular, $(u,p) \in C^1([0,T];L^2(0,1) \times L^2(0,1)) \cap C^0([0,T];H^1(0,1) \times H_0^1(0,1))$, and
$$
\max_{0 \le t \le T} \|\dt u(t)\|^2 + \|\dt p(t)\|^2 \le C \big( \|\dt u(0)\|^2 + \|\dt p(0)\|^2\big).
$$
The constant $C$ only depends on the bounds for the parameter $a$ and the time horizon $T$.
\end{lemma}
\begin{proof}
The assertions follow from standard results in semi-group theory \cite{DL5,Pazy83}.
\end{proof}
Problems with time independent right hand side and boundary condition can be reduced to the homogeneous case by subtracting a solution of the following stationary problem 
\begin{align}
\dx \bar p + a \bar u &= \bar f \qquad \text{in } (0,1),  \label{eq:stat1} \\
           \dx \bar u &= \bar g \qquad \text{in } (0,1).  \label{eq:stat2}
\end{align}
with boundary conditions 
\begin{align}
  \bar p &= \bar h \qquad \text{on } \{0,1\}. \label{eq:stat3}
\end{align}
We use a bar symbol to denote functions that are independent of time. 
A similar problem will arise later in the stability analysis for the time-dependent problem. 
In principle, the solution could be computed analytically here. Instead,
we use an argument for the well-posedness of the stationary problem that can be generalized to multiple dimensions.
\begin{lemma} \label{lem:stat}
Let $0 < a_0 \le a(x) \le a_1$. 
Then for any $\bar f$, $\bar g \in L^2(0,1)$, and $\bar h \in \RR^2$ 
the problem \eqref{eq:stat1}--\eqref{eq:stat3} has a unique strong solution 
$(\bar u,\bar p) \in H^1(0,1) \times H_0^1(0,1)$
and 
$$
\|\bar u\|_1 + \|\bar p\|_1 \le C \big( \|\bar f\| + \|\bar g\| + |\bar h|\big)
$$ 
with a constant $C$ only depending on the bounds $a_0$ and $a_1$. 
\end{lemma}
\begin{proof}
From the first equation, we get $\bar u = \frac{1}{a} (\bar f - \dx \bar p)$.
Inserting this into the second equation yields a second order elliptic problem for the pressure. 
Existence and uniqueness of a solution $\bar p \in H_0^1(0,1)$ 
and the a-priori estimate for $\bar p$ then follow from the Lax-Milgram theorem.
The result for $\bar u$ can finally be deduced from \eqref{eq:stat1} and \eqref{eq:stat2}.
\end{proof}

\section{Energy estimates} \label{sec:energy}

For a detailed stability analysis of the damped wave system, which is one focus of our investigations,
we will make use of the following family of generalized energies
\begin{align} \label{eq:energy}
\Ek(t) = \frac{1}{2} \big( \|\dtk u(t)\|^2 + \|\dtk p(t)\|^2 \big), \qquad k \ge 0.
\end{align}
As outlined in the introduction, some of these energies may have a physical interpretation, depending on the application context. 
The following estimates are standard and follow easily for sufficiently smooth solutions. 
To clarify the regularity requirements, we make a detailed statement and provide a short proof.
\begin{lemma}[Energy-identity]  \label{lem:ee1} $ $\\
Let $a \in L^\infty(0,1)$ and let $(u,p)$ be a solution of \eqref{eq:instat1}--\eqref{eq:instat3}
with finite energy $\Ek(0)$.
Then 
\begin{align}
\Ek(t) = \Ek(s) - \int_s^t \int_0^1 a(x) \; |\dtk u(x,r)|^2 dx\; dr, \qquad 0 \le s \le t \le T.
\end{align}
If in addition $a \ge 0$, then the respective energies decay monotonically.
\end{lemma}
\begin{proof}
Let us first prove the estimate for $k=0$  under the assumption that $(u,p)$ is a classical solution. 
Then $\E^0(t)$ is continuously differentiable, and we have 
\begin{align*}
\frac{d}{dt} \E^0(t) 
&= (\dt u(t),u(t)) + (\dt p(t), p(t))  \\
&= -(\dx u(t), p(t)) - (\dx p(t), u(t)) - (a u(t), u(t)) \\
&= -( a u(t), u(t)).  
\end{align*}
In the last transformation, we used integration-by-parts for the term $(\dx u(t),p(t))$ and the boundary conditions \eqref{eq:instat3} for the pressure.
The first energy identity now follows by integration over time. 
Since any mild solution can be approximated by classical solutions, the estimate
extends to all mild solutions by continuity.
Now let $k \ge 1$ and assume that $\Ek(0)$ is finite. Then 
$(v,q) = (\dtk u,\dtk p)$ is a mild solution of the hyperbolic system \eqref{eq:instat1}--\eqref{eq:instat3}, and we may apply the result for $k=0$. 
\end{proof}

\begin{remark}
The energy identities reveal that regularity, but also lack of regularity, is preserved for all time, 
which is a manifestation of the hyperbolic character of the problem.
The identity for $\E^1(t)$ therefore requires at least a classical solution,
while the identity for $\E^0(t)$ holds for all mild solutions as well.
As a consequence of the theorem, one can see that the respective energies are continuously differentiable, 
whenever they are finite. 
\end{remark}

As a next step of our stability analysis, 
we now show that the energies decay exponentially, 
provided that the damping is effective everywhere in the domain. 
\begin{theorem}[Exponential decay of the energy] \label{thm:ee2} $ $\\
Let $(u,p)$ be a solution of \eqref{eq:instat1}--\eqref{eq:instat3}
with finite energy $E^k(0) < \infty$. 
Moreover, assume that $a_0 \le a(x) \le a_1$ for some constants $a_0,a_1>0$.  
Then 
$$
\Ek(t) \le 3 e^{-\alpha (t-s)} \Ek(s), \qquad  0 \le s \le t \le T,
$$ 
with decay rate 
$\alpha = \frac{4}{3} a_0^3/(8 a_0^2+ 4 a_0^2 a_1 + 2 a_0 a_1  + a_1^4)$.
%
\end{theorem}
\begin{proof}
The proof is rather technical and therefore presented in detail in the appendix. 
Let us however sketch the basic steps already here.
We start with the case $k=1$.
Similar as in \cite{Zuazua88}, we define for $\eps>0$ the modified energies
\begin{align}
\E^1_\eps(t) = \E^1(t) + \eps (\dt u(t), u(t)).
\end{align}
As proven in Lemma~\ref{lem:a3}, the two energies $\E(t)$ and $\E_\eps(t)$ are equivalent for 
sufficiently small $\eps$, more precisely, $\frac{1}{2} \E^1(t) \le \E^1_\eps(t) \le \frac{3}{2} \E^1(t)$.
Under a further restriction on $\eps$, one can then show that
$\frac{d}{dt} \E^1_\eps(t) \le -\frac{2}{3}\eps \E^1_\eps(t)$, from which the result for $k=1$ follows; 
see Lemma~\ref{lem:a4}, where the precise definition of $\eps$ and $\alpha$ is given.
The case $k = 0$ can be deduced from the one for $k=1$ by an explicit construction,
see Section~\ref{sec:a3}.
The estimate for $k \ge 2$ finally follows by applying the result for $k=0$ to 
the functions $(\dtk u, \dtk p)$. 
\end{proof}
\begin{remark} \label{rem:asymptrate}
Let us define $c=a_0/a_1$. Then the decay rate can be expressed as 
\begin{align*}
\alpha = \frac{4}{3} \frac{c^2 a_0}{8 c^2 + 2 c + 4 c^2 a_1 + a_1^2} = \frac{4}{3} \frac{c^3}{ 8c^2/a_1 + 2c/a_1 + 4c^2 + a_1}. 
\end{align*}
From the first expression, one can see that $\alpha \ge c' a_0$ whenever $a_1 \le 1$, 
and from the second form we may deduce that $\alpha \ge c'' /a_1$ for $a_1 \ge 1$.  
In summary, we thus obtain $\alpha \ge  \min\{c' a_0,c''/a_1\}$ with constants $c',c''$ only depending on the ratio $a_0/a_1$. 
This estimate shows the correct dependence on the absolute size of the damping parameter, which can be verified analytically for the case of a constant damping parameter; see e.g. \cite{CoxZuazua94} and also compare to the numerical results presented in Section~\ref{sec:num}.
\end{remark}

\section{Variational formulations} \label{sec:var}

For the design and the analysis of appropriate discretization schemes, 
it will be useful to characterize the solutions of the damped wave system 
via variational principles which are suitable for a systematic approximation.

\subsection{Weak formulation of the stationary problem}
Testing \eqref{eq:stat1} and \eqref{eq:stat2} with appropriate test functions, 
using integration-by-parts for the first equation and the boundary conditions \eqref{eq:stat3}, 
we arrive at the following weak form of the stationary problem. For ease of presentation, 
we set $\bar h=0$ here.
\begin{problem} \label{prob:stat}
Find $\bar u \in H^1(0,1)$ and $\bar p \in L^2(0,1)$ such that 
\begin{align*}
-(\bar p, \dx \bar v) + (a \bar u, \bar v) &= (\bar f, \bar v) && \text{for all } \bar v \in H^1(0,1) \\
(\dx \bar u, \bar q) &= (\bar g, \bar q) && \text{for all } \bar q \in  L^2(0,1).
\end{align*}
\end{problem}
\noindent 
As before, we use the bar symbol to emphasize that the functions are independent of time. The existence of a unique weak solution follows almost directly from Lemma~\ref{lem:stat}. 
\begin{lemma} \label{lem:stat_weak}
Let $0<a_0 \le a(x) \le a_1$. Then for any $\bar f$, $\bar g \in L^2(0,1)$, Problem~\ref{prob:stat} has a unique solution which coincides with the strong solution of problem~\eqref{eq:stat1}--\eqref{eq:stat3} with $\bar h=0$.
As a consequence, we have $\|\bar u\|_1 + \|\bar p\|_1 \le C ( \|\bar f\| + \|\bar g\|)$ with $C$ depending on $a_0$ and $a_1$.
\end{lemma}
\begin{proof}
Existence of a weak solution follows from Lemma~\ref{lem:stat}. 
Now let $(\bar u, \bar p)$ be a weak solution. Then the first equation implies that $\bar p$ is weakly differentiable 
and therefore a strong solution, which was shown to be unique.  
\end{proof}

\begin{remark}
The well-posedness could be established here also via the abstract theory for 
mixed variational problems \cite{Brezzi74}. 
These kind of arguments will be utilized later for analyzing corresponding Galerkin approximations. 
\end{remark}

\subsection{Weak form of the instationary problem}
With a similar derivation as for the stationary problem above, 
one arrives at the following weak formulation for the time-dependent damped wave system.
\begin{problem} \label{prob:instat}
Find $(u,p) \in L^2(0,T;H^1(0,1) \times L^2(0,1)) \cap H^1(0,T;H^1(0,1)' \times L^2(0,1))$ 
with initial values $u(0)=u_0$ and $p(0)=p_0$, such that for a.e. $t \in (0,T)$ there holds
\begin{align*}
(\dt u(t), \bar v) - (p(t), \dx \bar v) + (a u(t), \bar v) &= 0 && \text{for all } \bar v \in H^1(0,1)\\
(\dt p(t), \bar q) + (\dx u(t), \bar q) &= 0 && \text{for all } \bar q \in L^2(0,1).
\end{align*}
\end{problem}
\begin{remark}
Here $H^1(0,1)'$ denotes the dual space of $H^1(0,1)$ and $(\dt u(t),\bar v)$ has to be understood as 
the duality product on $H^1(0,1)' \times H^1(0,1)$, accordingly.
The function spaces are chosen, such that all terms in the 
variational principle make sense. 
\end{remark}

Similar as before, the well-posedness of the weak formulation can be deduced from the previous results.
about existence of classical solutions.
\begin{lemma} \label{lem:instat_weak}
For any $u_0 \in H^1(0,1)$ and $p_0 \in H_0^1(0,1)$ the above weak formulation has a unique solution $(u,p)$
which coincides with the classical solution of problem \eqref{eq:instat1}--\eqref{eq:instat4}. 
In particular, the a-priori estimates of Lemma~\ref{lem:instat} are valid.  
\end{lemma}
\begin{proof}
Existence of a weak solution follows by Lemma~\ref{lem:instat}. 
Any weak solution also is a mild solution, which was shown to be unique; this yields the uniqueness. 
\end{proof}

\begin{remark}
Let us emphasize that either solution component can only be regular, if both initial data are regular. 
This can easily be seen from analytical solutions for the case of a constant damping parameter \cite{CoxZuazua94}.
Therefore, the regularity of both initial values is already required to show the well-posedness of the weak formulation.
Problem~\ref{prob:instat} thus actually provides a weak charaterization of classical solutions.
\end{remark}

\section{Galerkin semi-discretization} \label{sec:semi}

For the discretization in space, we consider Galerkin approximations for the weak formulations stated in the previous section based on approximation spaces
\begin{align*}
 V_h \subset H^1(0,1) \qquad \text{and} \qquad Q_h \subset L^2(0,1),
\end{align*}
which are assumed to be finite dimensional without further mentioning.
We start with considering the stationary case and then turn to the time dependent problem.

\subsection{Discretization of the stationary problem}

The Galerkin approximation of the weak form of the stationary system given in Problem~\ref{prob:stat} reads
\begin{problem} \label{prob:stath}
Find $\bar u_h \in V_h \subset H^1(0,1)$ and $\bar p_h \in Q_h \subset L^2(0,1)$ such that 
\begin{align*}
 - (\bar p_h, \dx \bar v_h)+(a \bar u_h, \bar v_h) &= (\bar f, \bar v_h) && \text{for all }
 \bar v_h \in V_h \\
 (\dx \bar u_h, \bar q_h) &= (\bar g, \bar q_h) && \text{for all } \bar q_h \in Q_h.
\end{align*}
\end{problem}
In order to establish the well-posedness of the Galerkin discretization, some compatibility conditions for the approximation spaces are required. We have
\begin{lemma} \label{lem:stath}
Let $0 < a_0 \le a(x) \le a_1$ and $Q_h \subset \dx V_h$. 
Then for any $\bar f, \bar g \in L^2(0,1)$, Problem~\ref{prob:stath} has a unique solution 
$(\bar u_h,\bar p_h)$ and $\|\bar u_h\|_1 + \|\bar p_h\| \le C \big( \|\bar f\| + \|\bar g\| \big)$ with a constant $C$ that only depends on the bounds $a_0$ and $a_1$.
\end{lemma}
\begin{proof}
Define $b(\bar v_h, \bar q_h) = (\dx \bar v_h,\bar q_h)$.
Then, by choosing $\bar v_h(x) = \int_0^x q_h(s) ds$, we get
\begin{align*}
\sup_{\bar v_h \in V_h} \frac{b(\bar v_h,\bar q_h)}{\|\bar v_h\|_1} 
\ge \frac{(\dx \int_0^\cdot \bar q_h,\bar q_h)}{\|\int_0^\cdot \bar q_h\|_1}   
\ge \frac{1}{\sqrt{2}}\|\bar q_h\| \qquad \text{for all } \bar q_h \in Q_h.
\end{align*}
Next, observe that $N_h=\{ \bar v_h \in V_h: b(\bar v_h,\bar q_h)=0\} = \{\bar v_h \in V_h: \dx \bar v_h = 0\}$. 
Hence 
\begin{align*}
a(\bar v_h,\bar v_h) = (a \bar v_h,\bar v_h) \ge a_0 \|\bar v_h\|^2 = a_0 \|\bar v_h\|_1^2  \qquad \text{for all } \bar v_h \in N_h.
\end{align*}
The assertions then follow from Brezzi's splitting lemma~\cite{Brezzi74}.
\end{proof}

\begin{remark}
The Galerkin approximation of the stationary problem defines a linear mapping $R_h:(\bar u, \bar p) \mapsto (\bar u_h,\bar p_h)$ usually called  \emph{elliptic projection} or \emph{Ritz projector}. 
This operator is frequently used in the error analysis of approximation schemes for the time-dependent problem \cite{Dupont73,Wheeler73}. 
We will later use instead the $L^2$-projections $\pi_h : L^2(0,1) \to V_h$ and $\rho_h: L^2(0,1) \to Q_h$ defined by
\begin{align*}
(\pi_h \bar  u, \bar v_h) &= (\bar u,\bar v_h) \qquad \text{for all } \bar v_h \in V_h \\
(\rho_h \bar p, \bar q_h) &= (\bar p,\bar q_h) \qquad \text{for all } \bar q_h \in Q_h.
\end{align*}
An advantage of the $L^2$-projections is that less regularity of the solution will be required for establishing convergence of the method; see Remarks~\ref{rem:regularity} and \ref{rem:regularity2} below and let us also refer to  \cite{Baker76,JenkinsRiviereWheeler02} for related ideas. 
\end{remark}

\subsection{Galerkin approximation of the instationary problem}
Recall that $V_h \subset H^1(0,1)$ and $Q_h \subset L^2(0,1)$ are assumed to be 
finite dimensional subspaces. The semi-discretization of the weak form for the time dependent problem 
then reads
\begin{problem} \label{prob:instath}
Find $(u_h, p_h) \in H^1(0,T;V_h \times Q_h)$ 
with $u_h(0) = \pi_h u_0$ and $p_h(0)=\rho_h p_0$, 
such that for a.e. $t \in (0,T)$ there holds
\begin{align*} 
 (\dt u_h(t), \bar v_h) - (p_h(t), \dx\bar v_h) + (a u_h(t), \bar v_h) &= 0 && \text{for all } \bar v_h \in V_h \\
(\dt p_h(t), \bar q_h) + (\dx u_h(t), \bar q_h) &= 0 && \text{for all } \bar q_h \in Q_h.
\end{align*}
\end{problem}
By choosing some bases for the spaces $V_h$ and $Q_h$, this system can be transformed into a linear ordinary differential equation, and we obtain
\begin{lemma} \label{lem:instath}
Let $a \in L^\infty(0,1)$.
Then for any $u_0,p_0 \in L^2(0,1)$, 
Problem~\ref{prob:instath} has a unique solution $(u_h,p_h)$ and 
$\|u_h(t)\|^2 + \|p_h(t)\|^2 \le C\big( \|u_0\|^2 + \|p_0\|^2 \big)$ for all $0 \le t \le T$ with constant $C$ only depending on the bounds for $a$ and the time horizon $T$.
\end{lemma}
\begin{proof}
Existence and uniqueness follow from the Picard-Lindelöf theorem, 
and the a-priori estimate follows from the energy identities given below. 
\end{proof}

\subsection{Discrete energy estimates}
We now conduct a refined stability analysis of the Galerkin approximations.
Let $(u^h,p^h)$ denote a solution of Problem~\ref{prob:instath}. 
Proceeding in a similar manner as on the continuous level, we define the semi-discrete generalized energies 
\begin{align*}
\Ek_h(t) = \frac{1}{2} \big( \|\dtk u_h(t)\|^2 + \|\dtk p_h(t)\|^2\big), \qquad k \ge 0.
\end{align*}
The following energy identities then follow almost directly from the special form of the variational principle underlying and the Galerkin approximation.
\begin{lemma} \label{lem:ee1h}
Let $a \in L^\infty(0,1)$ and $(u_h,p_h)$ be a solution of Problem~\ref{prob:instath}.
Then 
\begin{align*}
\Ek_h(t) = \Ek_h(s) - \int_s^t \int_0^1 a(x) |\dtk u_h(x,r)|^2 dx \; dr
\end{align*}
If $a \ge 0$, then the semi-discrete energies are monotonically decreasing.
\end{lemma}
\begin{proof}
Since the right hand side is zero, the discrete solution $(u_h,p_h)$ is always infinitely differentiable with respect to time. For $k=0$ we then obtain 
\begin{align*}
\frac{d}{dt} \E^0_h(t) 
&= (\dt u_h(t), u_h(t)) + (\dt p_h(t), p_h(t)) 
 = - (a u_h(t), u_h(t)),
\end{align*}
which follows directly by testing the variational principle with $\bar v_h = u_h(t)$ and $\bar
q_h = p_h(t)$.
The result for $k=0$ then follows by integration over time. 
The case $k \ge 1$ can be reduced to that for $k=0$ 
by observing that $(\dtk u_h,\dtk p_h)$ also solves the discrete problem.
\end{proof}

Under a mild compatibility condition for the approximation spaces $V_h$ and $Q_h$, 
we can also prove exponential decay estimates for the discrete energies. 
\begin{theorem}[Semi-discrete exponential stability] \label{thm:ee2h} $ $\\
Let $0<a_0 \le a(x) \le a_1$ and assume that $Q_h = \dx V_h$ and $1 \in V_h$.\\
Then any solution $(u_h,p_h)$ of Problem~\ref{prob:instath} satisfies
\begin{align*}
\Ek_h(t) \le 3 e^{-\alpha (t-s)} \Ek_h(s)
\end{align*}
with decay rate $\alpha>0$ that can be chosen as on the continuous level.
\end{theorem}
\begin{proof}
The proof follows almost verbatim as that of Theorem~\ref{thm:ee2} and is given in the appendix.
Note that the conditions $1 \in V_h$ and $Q_h = \dx V_h$ are required for the proof 
of the discrete analogue of Lemma~\ref{lem:a2}, which is then used in Lemma~\ref{lem:a3} 
and \ref{lem:a4} again. The proof of the exponential stability thus strongly relies on these
compatibility conditions.
\end{proof}

\subsection{An inhomogeneous problem}

The Galerkin semi-discretization can be generalized to problems with non-trivial boundary conditions or right hand sides. 
A problem of the following form will arise in the error analysis later on.
\begin{problem} \label{prob:instath2}
Find $(w_h,r_h) \in H^1(0,T;V_h \times Q_h)$ with $w_h(0) = 0$ and $r_h(0)=0$
such that for a.e. $t \in (0,T)$ there holds
\begin{align*}
 (\dt w_h(t), \bar v_h) - (r_h(t), \dx\bar v_h) + (a w_h(t), \bar v_h) &= (f(t),\bar v_h) &&
 \text{for all } \bar v_h \in V_h \\
(\dt r_h(t), \bar q_h) + (\dx w_h(t),\bar q_h) &= (g(t),\bar q_h) && \text{for all } \bar q_h
\in Q_h. 
\end{align*}
\end{problem}
Existence of a unique solution follows again from the Picard-Lindelöf theorem,
and the discrete energy estimates for the homogeneous problem yield uniform bounds.
\begin{lemma} 
\label{lem:instath2}
Let $0<a_0 \le a(x) \le a_1$ and assume that $Q_h = \dx V_h$ and $1 \in V_h$.\\
Then for any $f,g \in L^2(0,T;L^2(0,1))$,  Problem~\ref{prob:instath2} has a unique solution
$(w_h,r_h)$ and
\begin{align*}
\|w_h(t)\|^2 + \|r_h(t)\|^2 \le 3 \int_0^t e^{-\alpha(t-s)} \big( \|f(s)\|^2 + \|g(s)\|^2\big)
ds.
\end{align*}
\end{lemma}
\begin{proof}
By the variation of constants formula, the solution $Y_h=(w_h,r_h)$ of the discrete problem 
can be expressed as $Y_h(t)=\int_0^t e^{-A(t-s)} F(s) ds $ with $A$ and $F$ denoting the 
operator and right hand side governing the discrete evolution. 
From the discrete energy estimates of Theorem~\ref{thm:ee2h}, we deduce that 
$\|e^{-A(t-s)}\|^2 \le 3 e^{-\alpha(t-s)}$, from which the assertion follows by elementary calculations.
\end{proof}

\subsection{Discretization error estimates}

A-priori error estimates for the Galerkin semi-discretization can be obtained with 
standard arguments \cite{Dupont73,DouglasDupontWheeler78} 
and some modifications in order to advantage of the exponential stability. 
%
\begin{theorem}[Error estimates for the semi-discretization] \label{thm:eegalerkin} $ $\\
Let $0<a_0 \le a(x) \le a_1$ and assume that $Q_h = \dx V_h$ and $1 \in V_h$.
Then 
\begin{align*}
&\|u(t) - u_h(t)\|^2 + \|p(t) - p_h(t)\|^2 
  \le \|u(t) - \pi_h u(t)\|^2 + \|p(t) - \rho_h p(t)\|^2  \\
&  \qquad  \qquad \qquad + 3 \int_0^t e^{-\alpha(t-s)} \big( a_1^2 \|u(s) - \pi_h u(s)\|^2 + \|\dx (u(s) - \pi_h u(s)) \|^2 \big) ds. 
\end{align*} 
\end{theorem}
\begin{proof}
Following \cite{Dupont73}, we use a splitting of the error of the form
\begin{align*}
&\|u(t)-u_h(t)\| + \|p(t)-p_h(t)\| \\ 
& \qquad \le \big(\|u(t)-\pi_h u(t)\| + \|p(t)-\rho_h p(t)\|\big) + \big( \|\pi_h u(t) - u_h(t)\| + \|\rho_h p(t) - p_h(t)\|\big)
\end{align*}
into an approximation part and a discrete part. 
The first term already appears in the final estimate. 
To bound the second term, 
let us set $w_h=\pi_h u - u_h$ and $r_h = \rho_h p - p_h$. 
Then by definition of the discrete solution, we have  $w_h(0)=0$ and $r_h(0)=0$. 
Using the variational characterization of the discrete and the continuous solution, we further get
\begin{align*}
(\dt w_h(t)&, \bar v_h) - (r_h(t), \dx \bar v_h) + (a w_h(t), \bar v_h ) \\
&=(\dt \pi_h u(t) - \dt u_h(t), \bar v_h) - (\rho_h p(t) - p_h(t), \dx \bar v_h) + (a \pi_h u(t) - a u_h(t), \bar v_h ) \\
&=(\dt \pi_h u(t) - \dt u(t), \bar v_h) - (\rho_h p(t) - p(t), \dx \bar v_h) + (a \pi_h u(t) - a u(t), \bar v_h ) \\
&=( a \pi_h u(t) - a u_h(t), \bar v_h ).
\end{align*}
Here we used the properties of the projections $\pi_h$ and $\rho_h$, and the condition $\dx V_h  \subset Q_h$ in the last step.
In a similar manner, we obtain
\begin{align*}
(\dt r_h(t)&, \bar q_h) + (\dx w_h(t), \bar q_h)  \\
&=(\dt \rho_h p(t) - \dt p_h(t), \bar q_h) + (\dx \pi_h u(t) - \dx u_h(t), \bar q_h)  \\
&=(\dt \rho_h p(t)  - \dt p(t), \bar q_h) + (\dx \pi_h u(t) - \dx u(t), \bar q_h) \\
&=(\dx \pi_h u(t) - \dx u(t), \bar q_h),
\end{align*}
where the last step uses the properties of the projection $\rho_h$.
This shows that $(w_h,r_h)$ solves Problem~\ref{prob:instath2} with $f= a (\pi_h u - u)$ and $g=\dx \pi_h u(t) - \dx u(t)$. 
The estimate for the discrete part in the error splitting now follows from the bounds of Lemma~\ref{lem:instath2}.
\end{proof}

\begin{remark} \label{rem:regularity}
As typical for hyperbolic problems, the error estimate is slightly sub-optimal concerning the regularity requirements. 
Note, however, that the right hand side is bounded whenever $(u,p)$ is a weak solution and as long as $\|\dx \pi_h u\| \le C \|u\|_1$, i.e., when the $L^2$-projection $\pi_h$ is stable on $H^1(0,1)$. 
Similar error estimates can be obtained, if the elliptic projection is used in the error splitting \cite{Dupont73,DouglasDupontWheeler78}. 
The terms under the integral would then read $\big(\|\dt u(s) - \tilde \pi_h \dt u(s)\|^2 + \|\dt p(s) - \tilde \rho_h \dt p(s)\|^2\big)$. Since the elliptic projection requires some spatial regularity, the right hand side is finite only, if $\dt u(s)$ and $\dt p(s)$ have some spatial regularity,
which is not the case for all classical solutions. 
The use of the $L^2$-projection here thus yields the most general error estimate here.
\end{remark}


\section{A mixed finite element method} \label{sec:mixed}

A particular example of approximation spaces that satisfy the assumptions of the previous section are given by mixed finite elements. Let $T_h$ be a partition of the domain $(0,1)$
in subintervals of length $h$. We denote by $P_k(T_h)$ the space of piecewise polynomials of order $k$. 
One can now easily define pairs of compatible spaces of arbitrary order
with good approximation and stability properties.
\begin{lemma} \label{lem:approx}
Let $V_h = P_{k+1}(T_h) \cap H^1(0,1)$ and $Q_h = P_{k}(T_h)$ with $k \ge 0$. 
Then 
\begin{itemize}
 \item[(i)]    $Q_h = \dx V_h$ and $1 \in V_h$.
 \item[(ii)]   $\|\dx \pi_h u\| \le c_1 \|\dx u\|$ for all $u \in H^1(0,1)$.
\end{itemize}
Moreover, the following approximation properties hold
\begin{itemize}
 \item[(iii)] $\|p - \rho_h p\| \le c_3 h^{\min\{k+1,r\}} \|\dx^r p\|$ for any $p \in H^r(T_h) $.
 \item[(iv)]  $\|u - \pi_h u\| \le c_2 h^{\min\{k+2,r\}} \|\dx^r u\|$ for any $u \in H^r(T_h) \cap H^1(0,1)$.
 \item[(v)]   $\|\dx(u - \pi_h u)\| \le c_2 h^{\min\{k+1,r-1\}} \|\dx^r u\|$ for any $u \in H^r(T_h) \cap H^1(0,1)$.
\end{itemize}
The constants $c_i$ in the above estimates only depend on the polynomial degree $k$.
\end{lemma}
\begin{proof}
The first assertion follows directly from the construction,
and the stability estimate (ii) as well as the approximation error estimates are well known; see e.g. \cite{Schwab98}.
\end{proof}

As a direct consequence of Theorem~\ref{thm:eegalerkin} and Lemma~\ref{lem:approx}, we obtain 
\begin{theorem}[Mixed finite element approximation] \label{thm:eemixed} 
Let $V_h = P_{k+1}(T_h) \cap H^1(0,1)$ and $Q_h=P_k(T_h)$.
Moreover, assume that $0<a_0 \le a(x) \le a_1$. Then 
\begin{align*}
&\|u(t) - u_h(t)\|^2 + \|p(t) - p_h(t)\|^2  \\
&\le C h^{2 \min\{r,k+1\}} \Big( \|\dx^r u(t)\|^2 + \|\dx^r p(t)\|^2 + \int_0^t e^{-\alpha (t-s)} \big\{ \|\dx^{r} u(s)\|^2 + \|\dx^{r+1} u(s)\|^2 \big\} dx \Big), 
\end{align*} 
with constant $C$ depending only on the polynomial degree of approximation. 
\end{theorem}

\begin{remark} \label{rem:regularity2}
From the estimate with $r=0$, one can see that the error is bounded uniformly, 
whenever $(u,p)$ is a classical solution. 
This allows to show that the mixed finite element solution $(u_h,p_h)$ converges to $(u,p)$
in $L^p(0,\infty;L^2 \times L^2)$ as $h \to 0$ for all classical solutions.
%
%
%
The regularity conditions required to obtain quantitative estimate can be verified under the assumption that 
$a$ is sufficiently smooth and that the initial values are smooth and satisfy the usual compatibility conditions \cite{Evans98}. 
From the estimates for the generalized energies stated in Theorem~\ref{thm:ee2}, one can deduce that
\begin{align*}
\|\dx^s p(t) \|^2 + \|\dx^s u(t)\|^2 \le C_r e^{-\alpha t}, \qquad 1 \le s \le r,
\end{align*}
with a constant $C_r$ independent of time. 
Together with Theorem~\ref{thm:eemixed}, we then obtain 
\begin{align*}
\|u(t) - u_h(t)\|^2 + \|p(t) - p_h(t)\|^2 \le C' h^{2\min\{k+1,r\}} (1+t) e^{-\alpha t}. 
\end{align*}
The error thus even converges exponentially to zero with $t \to \infty$. This is not surprising, 
since both, the continuous and the discrete solution, converge to zero exponentially here. 
If time-independent right hand sides or boundary conditions are prescribed, then 
the equilibrium of the system, which is described by the stationary problem, 
is not zero and one would obtain an additional term $C'' h^{2\min\{k+1,r\}}$ in the error estimate, accounting for the approximation error of the stationary problem.
\end{remark}

\section{Time discretization} \label{sec:time}

We now turn to the time discretization of the Galerkin schemes discussed in the previous section. 
Let $\tau>0$ be the time-step and set $t^n = n \tau$ for $n \ge 0$. 
Given a sequence $\{u_h^n\}_{n \ge 0}$, 
we further define the symbols
\begin{align*}
u_h^{n,\theta} &:= \theta u_h^{n} + (1-\theta) u_h^{n-1}, 
\qquad \text{as well as } \\
\bar\partial_\tau^0 u_h^n &:= u_h^n
\qquad \text{and} \qquad \bar 
\partial_\tau^{k+1} u_h^n :=  \frac{\bar \partial_\tau^k u_h^n - \bar \partial_\tau^k u_h^{n-1}}{\tau} \quad \text{for } k \ge 0.
\end{align*}
Therefore, $\bar\partial_\tau^k u_h^n$ corresponds to the $k$th backward difference quotient.
To mimick the notation on the continuous level, we also write $\bar\partial_{\tau\tau} u_h^n$ instead of $\bar\partial_\tau^2 u_h^n$.

\subsection{Fully discrete scheme}

As before, we assume that $V_h \subset H^1(0,1)$ and $Q_h \subset L^2(0,1)$ are some finite dimensional 
subspaces. 
For the time discretization of Problem~\ref{prob:instath}, we now consider the following family of fully discrete approximations. 
\begin{problem}[$\theta$-scheme] \label{prob:theta}$ $\\
Set $u_h^0=\pi_h u_0$ and $p_h^0=\rho_h p_0$. Then, for $n \ge 1$, find 
$(u_h^n,p_h^n)\in V_h\times Q_h$ such that 
\begin{align}
(\bar\partial_\tau u_h^n,\bar v_h)-(p_h^{n,\theta},\dx\bar v_h)+(a u_h^{n,\theta},\bar v_h)&=0 &&\text{for all }\bar v_h\in V_h \label{prob:theta1}\\
(\bar\partial_\tau p_h^n ,\bar q_h)+(\dx u_h^{n,\theta},\bar q_h)&=0 &&\text{for all }\bar q_h\in Q_h.\label{prob:theta2}
\end{align}
\end{problem}
\noindent
Note that the system \eqref{prob:theta1}--\eqref{prob:theta2} can be written equivalently as 
\begin{align*}
\tfrac{1}{\tau} (u_h^{n},\bar v_h)  - \theta (p_h^{n}, \dx \bar v_h) + \theta (a u_h^{n}, \bar v_h)
&= \tfrac{1}{\tau}  (u_h^{n-1},\bar v_h)  +(1-\theta) (p_h^{n-1}, \dx \bar v_h) 
\\ &\,\quad
-(1-\theta) (a u_h^{n-1}, \bar v_h)\\
\tfrac{1}{\tau} (p_h^{n}, \bar q_h) + \theta (\dx u_h^{n}, \bar q_h) &= \tfrac{1}{\tau} (p_h^{n-1}, \bar q_h) - (1-\theta) (\dx u_h^{n-1}, \bar q_h).
\end{align*}
The well-posedness of the problem of determining $(u_h^{n},p_h^{n})$ from $(u_h^{n-1},p_h^{n-1})$ can 
then be shown with the same arguments as used in Lemma~\ref{lem:stath}, and we obtain
\begin{lemma}
Let $0 \le a(x) \le a_1$ and $Q_h \subset \dx V_h$. 
Then for any $u_0,p_0\in L^2(0,1)$ and any $\theta \ge 0$, 
Problem~\ref{prob:theta} is well-posed and has a 
unique solution $( u_h^n,p_h^n)$ for all $n \ge 0$.
\end{lemma}
\noindent 
Uniform bounds for the solution can again be obtained via energy arguments.

\subsection{Discrete energy estimates}

For the stability analysis of the fully discrete problem, we utilize energy estimates similar as 
on the continuous and the semi-discrete level. Given a solution $\{(u_h^n,p_h^n)\}$ 
of Problem~\ref{prob:theta}, we define the discrete energies at time $t^n$ by
\begin{align*}
E_h^{k,n} = \frac{1}{2}\big( \|\bar\partial_\tau^k u_h^n\|^2 + \|\bar\partial_\tau^k p_h^n\|^2 \big), \qquad k \ge 0.
\end{align*}
By appropriate testing of the fully discrete scheme \eqref{prob:theta1}--\eqref{prob:theta2} and with 
similar arguments as on the continuous level, we now obtain the following energy identities.
\begin{lemma}[Discrete energy-identity] \label{lem:enbal} 
Let $\{(u_h^n,p_h^n)\}$ be a solution of Problem~\ref{prob:theta}.
Then 
\begin{align} \label{enbal}
\bar\partial_\tau E_h^{k,n} 
= -(\theta-\tfrac{1}{2}) \tau \big(\|\bar \partial_\tau^{k+1} u_h^{n}\|^2+\|\bar\partial_\tau^{k+1} p_h^{n}\|^2\big) 
- (a \bar \partial_\tau^k u_h^{n,\theta}, \bar\partial_\tau^k u_h^{n,\theta}).
\end{align}
\end{lemma}
\begin{proof}
The result for $k=0$ follows by setting $\bar v_h = u_h^{n,\theta}$ and $\bar q_h = p_h^{n,\theta}$ in
\eqref{prob:theta1}-\eqref{prob:theta2}. The estimate for $k \ge 1$ reduces to the one for $k=0$ by 
observing that, due to linearity of the problem, the differences $(\bar \partial_\tau^k u_h^n, \bar \partial_\tau^k p_h^n)$ 
solve the system \eqref{prob:theta1}-\eqref{prob:theta2} as well. 
\end{proof}

Observe that, without further arguments, a decay of the discrete energy can only be guaranteed, 
if we require $\theta \ge 1/2$.  
%
With similar arguments as already used for the analysis of the time-continuous problem, 
we can also establish the exponential decay of the energies for the fully discrete setting.
\begin{theorem}[Discrete exponential stability] \label{thm:ee2hn}
Let $0<a_0\le a(x)\le a_1$ and $\frac{1}{2} < \theta \le 1$.
Moreover, assume that $Q_h=\dx V_h$ and $1\in V_h$. 
Then for any $0<\tau \le \tau_0$ sufficiently small, there holds
\begin{align*}
E_h^{k,n} \le 3 e^{-\alpha (n-m) \tau} E_h^{k,m} \qquad \text{for all } k \le m \le n
\end{align*}
with decay rate $\alpha=\frac{2}{3}a_0^3/(8a_0^2+4a_0^2a_1+3a_0a_1+4a_1^4)$.
\end{theorem}
\begin{proof}
The proof follows with similar arguments as used for the time-continuous case; details are given again in the 
appendix. An upper bound for the maximal stepsize $\tau_0$ depending only on $a_0, a_1$, and $\theta$, 
is given in Lemma~\ref{lem:a7}.
\end{proof}


\begin{remark}[A second order method] \label{rem:adaptive1}
A careful inspection of the proof of Theorem~\ref{thm:ee2hn} reveals that the assertion of Theorem~\ref{thm:ee2hn} 
remains valid if one chooses $\theta=\frac{1}{2}+\lambda \tau$ with $\lambda$ sufficiently large; see the corresponding remarks in the appendix. As we will illustrate by numerical tests, the uniform and unconditional exponential stability gets lost, however, for the choice $\theta=1/2$ corresponding to the Crank-Nicolson scheme. 
\end{remark}

\begin{remark}
The above results hold unconditionally, i.e., without restriction on the time step size $\tau$. 
Without going into details, let us note that under strong restrictions on the time step size, 
uniform exponential stability can also be shown for the schemes with $0 \le \theta\le 1/2$. 
The reason for this is that the stability of the time-continuous problem dominates the discretization error 
in that case.
\end{remark}

\subsection{An inhomogeneous problem}

In the derivation of the error estimate, we will again arrive at a problem with inhomogeneous right hand sides. 

\begin{problem}[Inhomogeneous problem] \label{prob:thetainhom}
Let $f^{n}$, $g^{n}$, $n \ge 1$ be given and set $w_h^0=0$ and $r_h^0=0$.
Then, for $n \ge 1$, find $(w_h^n,r_h^n) \in V_h\times Q_h$ such that 
\begin{align}
(\bar\partial_\tau w_h^n,\bar v_h)-(r_h^{n,\theta},\dx \bar v_h)+(a w_h^{n,\theta},\bar v_h)&=(f^{n},\bar v_h) &&\text{for all }\bar v_h\in V_h \label{prob:thetainhom1}\\
(\bar\partial_\tau r_h^n,\bar q_h)+(\dx w_h^{n,\theta},\bar q_h)&=(g^{n},\bar q_h) &&\text{for all }\bar q_h\in Q_h.\label{prob:thetainhom2}
\end{align}
\end{problem}
Using the estimates for the homogeneous problem and a discrete version of the variation of constants formula, we arrive at the following a-priori bound.
\begin{lemma} \label{lem:instathn2}
Let the assumptions of Theorem~\ref{thm:ee2hn} hold.
Then for any $0<\tau\le\tau_0$ sufficiently small and any $\{f^{n}\},\{g^{n}\}\subset L^2(0,1)$, 
Problem~\ref{prob:thetainhom} has a unique solution and
\begin{align*}
 \|w_h^n\|^2+\|r_h^n\|^2 \le 3 \sum\nolimits_{i=1}^n \tau e^{-\alpha (n-i)\tau} \big(\|f^{i}\|^2+\|g^{i}\|^2\big).
\end{align*}
\end{lemma}
\begin{proof}
The result follows with similar arguments as that of Lemma~\ref{lem:instath2}. 
\end{proof}


\subsection{Error estimates}

We are now in the position to prove our main error estimate for the fully discrete scheme.
By carefully adopting standard arguments \cite{Dupont73} in order to take advantage of the discrete exponential stability, we obtain estimates uniform in time.
\begin{theorem}[Error estimate for the full discretization] $ $\\
Let the assumptions of Theorem~\ref{thm:ee2hn} hold and  $0 < \tau \le \tau_0$ be sufficiently small. Then
\begin{align*}
 &\| u^n- u_h^n\|^2+\| p^n-p_h^n\|^2 \le \| u^n-\pi_h u^n\|^2+\| p^n-\rho_h p^n\|^2\\
 &\qquad\qquad +6\sum\nolimits_{i=1}^n \tau e^{-\alpha (n-i)\tau} \big( a_1^2\|\pi_h u^{i,\theta}- u^{n,\theta}\|^2  + \|\dx(\pi_h u^{i,\theta}-u^{n,\theta})\|^2\big)\\
 &\qquad\qquad+ C \tau^2 \sum\nolimits_{i=1}^n \tau e^{-\alpha (n-i)\tau} \big(\|\dtt u(\xi^i)\|^2+\|\dtt p(\eta^i)\|^2 \big).
%
\end{align*}
with $\xi^i$, $\eta^i \in (t^{i-1},t^{i})$, 
and constant $C$ independent of $\tau$.
\end{theorem}
\noindent 
For ease of notation, we used the symbols $ u^n=u(t^n)$ and $ u^{n,\theta}=\theta u^n+(1-\theta)\u^{n-1}$ here.
\begin{proof}
 Similar as in the proof of Theorem \ref{thm:eegalerkin}, we use the error splitting
\begin{align*}
&\| u^n- u_h^n\| + \| p^n-p_h^n\| \\
&\qquad \le \big(\| u^n-\pi_h  u^n\| + \| p^n-\rho_h  p^n\|\big) + \big( \|\pi_h  u^n -  u_h^n\| + \|\rho_h
 p^n - p_h^n\|\big).
\end{align*}
The first term already appears in the final estimate. 
To bound the discrete error components $ w_h^n=\pi_h u^n- u_h^n$ and $ r_h^n=\rho_h p^n-p_h^n$, observe that
\begin{align*}
&(\bar\partial_\tau w_h^n,\bar v_h)-( r_h^{n,\theta},\dx\bar v_h)+(a w_h^{n,\theta},\bar v_h)
\\&\qquad \qquad \qquad 
=(a\pi_h u^{n,\theta}-a u^{n,\theta},\bar v_h)+ (\bar\partial_\tau u^n-\dt u^{n,\theta},\bar v_h),
\end{align*}
where we used the definition of the continuous and the fully discrete solution and the properties of the
projections $\pi_h$ and $\rho_h$. In a similar manner, we obtain
\begin{align*}
&(\bar\partial_\tau r_h^n,\bar q_h)+ (\dx w_h^{n,\theta},\bar q_h)
 +(\dx\pi_h u^{n,\theta}-\dx u_h^{n,\theta},\bar q_h)
\\ & \qquad \qquad \qquad 
=(\dx\pi_h u^{n,\theta}-\dx u^{n,\theta},\bar q_h)+(\bar\partial_\tau p^n-\dt p^{n,\theta},\bar q_h).
\end{align*}
The discrete error $( w_h^n, r_h^n)$ thus solves Problem~\ref{prob:thetainhom} with inhomogeneous 
right hand sides $f^{n}=(a\pi_h u^{n,\theta}-a u^{n,\theta})+(\bar\partial_\tau u^n-\dt u^{n,\theta})$ and $g^{n}=(\dx\pi_h u^{n,\theta}-\dx u^{n,\theta})+(\bar\partial_\tau p^n-\dt p^{n,\theta})$.
By elementary estimates, one further obtains 
$\|f^n\| \le a_1 \|\pi_h u^{n,\theta} - u^{n,\theta}\| + \tau \|\bar\partial_{\tau\tau} u(\xi^n)\|$
and also
$ \|g^n\| \le \|\dx \pi_h u^{n,\theta} - \dx u^{n,\theta}\| + \tau \|\bar\partial_{\tau\tau} p(\eta^n)\|$ 
with appropriate values of $\xi^n,\eta^n \in (t^{n-1},t^n)$. 
The assertion of the theorem then follows by the a-priori bounds of Lemma~\ref{lem:instathn2}. 
%
%
\end{proof}

\begin{remark}[A second order scheme] \label{rem:adaptive2} $ $\\
For the choice $\theta=\frac{1}{2}+\lambda \tau$, one can use Taylor estimates
$
\|\bar\partial_\tau u^n-\dt u^n\| 
=O(\tau^2)
$
and
$
\|\bar\partial_\tau p^n-\dt p^n\| 
=O(\tau^2)
$
that are second order in time. The estimate of the previous theorem thus holds with $\tau^2$ and second derivatives replaced by $\tau^4$ and third derivatives in the last line.
The choice of the parameter $\theta=\frac{1}{2}+\lambda \tau$ thus allows to obtain a second order time discretization scheme with unconditional and uniform exponential stability.
\end{remark}

\begin{remark}[Convergence rates] \label{rem:fullrate}
In combination with the mixed finite element Galerkin approximation
discussed in Section~\ref{sec:mixed}, the time discretization scheme $\theta > 1/2$ yields
\begin{align*}
\|u(t^n) - u_h^n\|^2 + \|p(t^n) - p_h^n\|^2 
\le  C e^{-\alpha t^n} \big( h^{2k+2} + \tau^2 \big),
\end{align*}
for some $\alpha>0$ depending only on $a_0$, $a_1$, and $\theta$,
provided that the solution $(u,p)$ is sufficiently smooth.
For the adaptive choice $\theta = 1/2+\lambda \tau$ with $\lambda$ sufficiently large, 
one can even guarantee that
\begin{align*}
\|u(t^n) - u_h^n\|^2 + \|p(t^n) - p_h^n\|^2 
\le  C e^{-\alpha t^n} \big( h^{2k+2} + \tau^4 \big)
\end{align*}
with rate $\alpha$ depending on $a_0$, $a_1$, and $\lambda$. 
All these estimates are unconditional, i.e., no condition on the time step size is needed, 
and they hold uniformly in time. 
%
\end{remark}

%
%
%
%
%
%
%
%
%

\section{Numerical validation} \label{sec:num}

We now illustrate our theoretical results with some numerical test. 
In order to allow for analytic solutions and to guarantee sufficient smoothness, 
we choose $a\equiv const>0$. 

\subsection{Exponential convergence}
Let us start by comparing the decay behaviour of the continuous and the discrete solutions. 
Using separation of variables, one can see that 
\begin{align} \label{eq:solu}
u(x,t)=\exp(-at/2+\sqrt{a^2/4-\pi^2}t)\cos(\pi x)
\end{align}
and 
\begin{align} \label{eq:solp}
p(x,t)=1/\pi(-a/2-\sqrt{a^2/4-\pi^2})\exp(-at/2+\sqrt{a^2/4-\pi^2}t)\sin(\pi x)
\end{align}
solves the damped wave system~\eqref{eq:instat1}--\eqref{eq:instat4} with $a\equiv const$.

For our numerical tests, we choose $a=10$ and compute the discrete solutions with the mixed finite element approximation with $P_1$-$P_0$ elements for the velocity and pressure, and using the $\theta$-scheme with $\theta=1$ (implicit Euler) and $\theta=\frac{1}{2}+\tau$ (second order).
In Table \ref{table:expdecay}, we report about the energy decay for the exact, the semi-discrete, and the fully discrete solutions obtained with discretization parameters $h=\tau=10^{-3}$.
\begin{table}[ht!]
\renewcommand{\arraystretch}{1.2}
\begin{center}
\begin{tabular}{c||c|c|c|c|c|c||c}
$t^n$                     & 0    & 2        & 4        & 6        & 8        & 10       & $\alpha$\\
\hline
\hline
exact                     & 2.25 & 2.65e-02 & 3.13e-04 & 3.69e-06 & 4.35e-08 & 5.12e-10 & 2.139\\
\hline
$\theta=1$                & 2.25 & 2.66e-02 & 3.14e-04 & 3.71e-06 & 4.39e-08 & 5.18e-10 & 2.138\\
\hline
$\theta=\frac{1}{2}+\tau$ & 2.25 & 2.65e-02 & 3.13e-04 & 3.69e-06 & 4.34e-08 & 5.12e-10 & 2.139\\
\end{tabular}
\smallskip
\caption{Decay of the exact energy $E^0(t^n)$ and the corresponding energies of the semi-discrete 
and fully discrete approximations obtained with the mixed finite element approximation combined with the implicit Euler method ($\theta=1$) and the second order scheme ($\theta = \frac{1}{2}+\tau$), respectively. 
}
\label{table:expdecay}
\end{center}
\end{table}

As predicted by our theoretical results, all energies decrease exponentially and approximately at the same rates.

\subsection{Non-uniform exponential stability of the Crank-Nicolson method}

As mentioned in Remark~\ref{rem:adaptive1}, 
the unconditional and uniform exponential stability for the fully discrete scheme can be guaranteed also 
for the choice $\theta=\frac{1}{2}+\lambda \tau$ with $\lambda$ sufficiently large, which yields a second order approximation in time.
For $\theta=1/2$, one obtains the Crank-Nicolson method, which is also formally second order accurate in time.
We will now demonstrate, that the uniform exponential stability is however lost without further restrictions on the size of the time step. 

As before, we set $a=10$.  
As initial values, we now choose $u_0=0$ and $p_0$ as the hat function on $[0,1]$,
which ensures that components of all spatial frequencies are present in the solution.
We then compute the numerical solutions with the mixed finite element approximation and 
the $\theta$-scheme with $\theta=\frac{1}{2}$ as well as $\theta=\frac{1}{2}+\tau$. 
For our tests, we use a fixed time step $\tau=10^{-2}$ and different mesh sizes $h=2^{-k}$ for some values of $k \ge 1$. 
The evolution of the discrete energies $E_h^n$ is depicted in Figure~\ref{fig:countercrank}.
\begin{figure}[ht!]
\begin{center}
\includegraphics[width=0.6\textwidth]{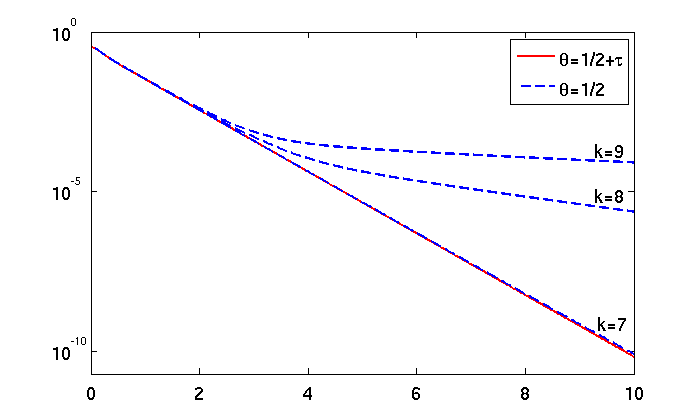}
\caption{Evolution of the discrete energies $E_h^n$ for the fully discrete solutions with $\theta=\frac{1}{2}$ (Crank-Nicolson, blue) and $\theta=\frac{1}{2}+\tau$ (second order, red) for fixed time step $\tau=10^{-2}$ and mesh size 
$h=2^{-k}$ with $k=7,8,9$.}
\label{fig:countercrank}
\end{center}
\end{figure}
As can be seen from the plots, the exponential stability of the Crank-Nicolson energy is lost when $h$ becomes 
much smaller than $\tau$ while the decay of the second order scheme with $\theta=\frac{1}{2}+\tau$ remains uniform. 
Let us remark at this point that the uniform exponential stability could be maintained also for the Crank-Nicolson scheme under a condition $\tau \le c h$ on the time step, see \cite{ErvedozaZuazua09} for results in this direction, and even for 
explicit Runge-Kutta methods with sufficiently small time steps.

\subsection{Asymptotic behavior of the decay rate}

Our theoretical results allow us to make some predictions about the dependence of the decay rate on 
the upper and lower bounds $a_0,a_1$ for the parameter $a$.
As indicated in Remark~\ref{rem:asymptrate}, one has $\alpha \ge \min\{c' a_0,c''/a_1\}$ 
with appropriate constants $c'$ and $c''$ only depending on the ratio $a_1/a_0$.  
For $a \equiv const$, even an analytic expression 
\begin{align*}
\alpha=g(a):= a/2-\text{Re}\sqrt{a^2/4-\pi^2}.
\end{align*} 
for the decay rate can be computed; see for instance \cite{CoxZuazua94}. 

We now illustrate that the correct behaviour of the decay rate is reproduced be the semi-discretization and 
the full discretization proposed in this paper. 
To do so, we compute the norms of $S_h(t^n) : (u_h(0),p_h(0)) \mapsto (u_h(t^n),p_h(t^n))$ and $S_h^\tau(t^n) : (u_h(0),p_h(0)) \mapsto (u_h^n,p_h^n)$ governing the semi-discrete and discrete evolutions, respectively. 
In our tests, we set $t^n=10$ and compute $\|S_h(t^n)\|$ for $h=10^{-1},10^{-2}$ and $\|S_h^\tau(t^n)\|$ with 
$h=\tau=1/20$ using the second order scheme with $\theta=\frac{1}{2}+\tau$. 
The damping parameter is chosen from the set $a=2^{-5},...,2^{10}$. 
The results of our numerical tests are depicted in Figure~\ref{fig:arate}.

\begin{figure}[ht!]
\begin{center}
\includegraphics[width=0.7\textwidth]{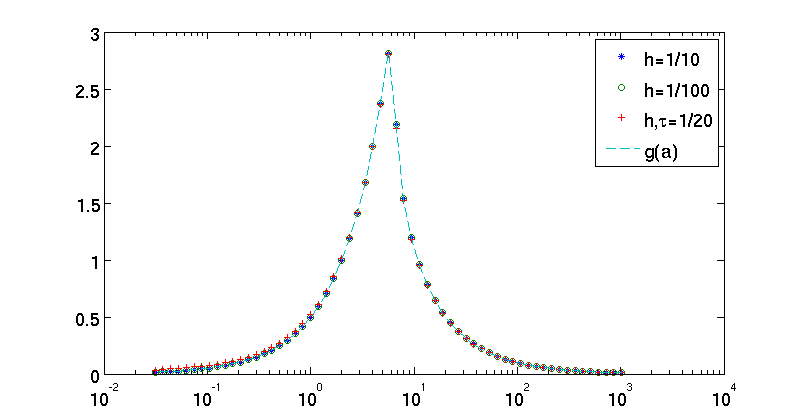}
\caption{Exponential decay rates for semi-discrete and fully discrete evolution operators $\|S_h(10)\|$ with $h=1/10$, $h=1/100$ and $S_h^\tau(10)$ with $h=\tau=1/20$ for damping parameters $a=2^{-k}$ with $k=-5,\ldots,10$.}
\label{fig:arate}
\end{center}
\end{figure}

Note that already for the coarse discretization, the numerically observed decay 
rates are in almost perfect agreement with the analytical formula, over a very 
large range of parameters $a$, which illustrate the robustness of our 
results with respect to the damping parameter. 

\subsection{Convergence rates}
Let us finally also report on the convergence of the discretization errors 
with respect to the mesh size $h$ and the time step $\tau$. 
As analytical solution, we choose the one given in \eqref{eq:solu}--\eqref{eq:solp} and we set $a=10$ 
as before. The discrete approximations $(u_h^n,p_h^n)$ are computed with the mixed finite element method with $P_1$-$P_0$ elements and the $\theta$-scheme with $\theta=\frac{1}{2}+\tau$ and $\theta=1$. 
As a measure for the discretization error, we choose the discrete error component
\begin{align*}
e_h^\tau = \big( \|u_h^n - \pi_h u(t^n)\|^2 + \|p_h^n - \rho_h p(t^n)\|^2 \big)^{1/2}.
\end{align*}
By a careful inspection of the proofs, one can see that the discrete error 
exhibits super-convergence with respect the spatial discretization. 
We thus expect $e_h^n = O(h^2 + \tau^p)$ with order $p=1$ for the implicit Euler method and $p=2$ for the second order scheme.

\begin{table}[ht!]
\begin{center}
\small
\begin{tabular}{l||c|c||c|c}
$h$     & $\theta=1$ & \text{rate}  & $\theta=\frac{1}{2}+\tau$ & \text{rate} \\ 
\hline      
\hline      
0.5     & 0.13747 & ---  &  0.13748 & ---  \\
0.25    & 0.03569 & 1.95 &  0.03570 & 1.95 \\
0.125   & 0.00894 & 2.00 &  0.00895 & 2.00 \\
0.0625  & 0.00223 & 2.00 &  0.00224 & 2.00 
\end{tabular}
\hspace*{2em}
%
%
\begin{tabular}{l||c|c||c|c}
$\tau$  & $\theta=1$ & \text{rate}  & $\theta=\frac{1}{2}+\tau$ & \text{rate}\\ 
\hline      
\hline      
0.5     & 0.17830 & ---  &  0.17830 & ---  \\
0.25    & 0.09741 & 0.87 &  0.04782 & 1.90 \\
0.125   & 0.05113 & 0.93 &  0.01216 & 1.98 \\
0.0625  & 0.02623 & 0.96 &  0.00305 & 1.99 
\end{tabular}
\bigskip
\caption{Convergence of the discrete error $e_h^\tau$ with respect to the mesh size $h$ (left, $\tau=10^{-5}$) and the time step $\tau$ (right, $h=10^{-4}$) for the mixed finite element approximation and $\theta$-scheme with $\theta=1$ and $\theta=1/2+\tau$.}
\label{table:convrate}
\end{center}
\end{table}

The results obtained in our numerical tests are displayed in Table \ref{table:convrate}.
Again they perfectly match the theoretical predictions already at very coarse discretization levels.

\section{Discussion}

In this paper, we considered the systematic numerical approximation of a damped wave system by 
Galerkin semi-discretization in space and time discretization by certain one-step methods. 
We derived energy decay estimates on the continuous level and showed that these remain valid 
uniformly for the semi-discretizations and fully discrete approximations under general 
assumptions on the approximation spaces and the parameter $\theta$ used for the time discretization.
Moreover, the estimates are unconditional, i.e., the time step $\tau$ can be chosen independently of the 
discretization spaces. 

While we only considered here a one-dimensional model problem, our results and methods of proof can 
in principle also be generalized to multi-dimensional problems and other applications having similar 
structure. Also non-linearities can be tackled to some point; we refer to \cite{ErvedozaZuazua09,Joly03} 
for some general analysis in this direction. 

\section*{Acknowledgements}
The authors would like to gratefully acknowledge the support by the German Research Foundation (DFG) via grants IRTG~1529, GSC~233, and TRR~154.

\appendix 

\renewcommand\thesection{A}
\renewcommand\thelemma{A.\arabic{lemma}}

\setcounter{lemma}{0}

\section*{Appendix}

\subsection{Auxilliary results}

We start with proving a generalized Poincar\'e inequality.

\begin{lemma} \label{lem:a1}
Let $a \in L^2(0,1)$ and $\bar a = \int a dx \ne 0$.
Then for any $u \in H^1(0,1)$ we have
\begin{align*}
\|u\|_{L^2(0,1)} \le  \frac{1}{\pi}\big( 1 + \frac{1}{\bar a} \|a-\bar a\|_{L^2(0,1)}  \big) \ \|\dx u\|_{L^2(0,1)} + \frac{1}{\bar a} \big| \int_0^1 a u \; dx\big|,
\end{align*} 
\end{lemma}
\begin{proof}
We denote by $\bar u = \int_0^1 u$ the average of $u$ and by $\|\cdot\|$ the norm of $L^2(0,1)$. 
Then 
\begin{align*} 
\|u\| \le \|u-\bar u\| + \|\bar u\| \le \frac{1}{\pi} \|\dx u\| + \|\bar u\|,
\end{align*}
where we used the standard Poincar\'e inequality.
To bound the last term, observe that 
\begin{align*}
\bar u 
&= \frac{1}{\bar a} \int_0^1 \bar a u \; dx
= \frac{1}{\bar a} \int_0^1 (\bar a -a) \; u + a u \; dx 
= \frac{1}{\bar a} \int_0^1 (\bar a -a) \; (u-\bar u) + a u \; dx.
\end{align*}
Application of the triangle, Cauchy Schwarz inequalities, and Poincar\'e inequality yields
\begin{align*}
\|\bar u\| 
&\le \frac{1}{\bar a} \|\bar a - a\| \|u - \bar u\| + \frac{1}{\bar a} |\int_0^1 a u \; dx| 
\le \frac{\|\bar a - a\|}{\bar a \pi} \|\dx u\| + \frac{1}{\bar a} |\int_0^1 a u \; dx|.
\end{align*}
The assertion of the lemma now follows by combination of the two estimates.
\end{proof}

An application of this lemma to solutions of the damped wave system yields the following estimate which will be used several times below.
\begin{lemma} \label{lem:a2}
Let $(u,p)$ be a classical solution of \eqref{eq:instat1}--\eqref{eq:instat3} 
and let $0 < a_0 \le a(x) \le a_1$. Then 
\begin{align*}
\|u(t)\|_{L^2(0,1)} \le \frac{a_1}{a_0} \|\dt p(t)\|_{L^2(0,1)} + \frac{1}{a_0} \big| \|\dt u(t)\|.
\end{align*}
\end{lemma}
\begin{proof}
Using the bounds for the parameter, we obtain from the previous lemma that
\begin{align*}
\|u\|_{L^2(0,1)} \le  \frac{a_1}{a_0} \|\dx u\|_{L^2(0,1)} + \frac{1}{a_0} \big| \int_0^1 a u \; dx\big|.
\end{align*}
Note that this estimate holds for any function $u \in H^1(0,1)$. 
Using the mixed variational characterization of the solution, we further obtain
\begin{align*}
|\int_0^1 a u \; dx| = |(au, 1)| = |-(\dt u, 1) + (p, \dx 1)| = |(\dt u, 1)| \le \|\dt u\|.
\end{align*}
Note that the boundary condition on the pressure was used implicitly here.
\end{proof}

\subsection{Proof of the Theorem~\ref{thm:ee2} for $k = 1$} \label{sec:est2}

To establish the decay estimate for the energy $\E^1(t) = \frac{1}{2} \big(\|\dt u(t)\|^2 + \|\dt p(t)\|^2 \big)$, let us define the modified energy 
$$
\E_\eps^1(t) = \E^1(t) + \eps (u_t(t),u(t)).
$$
We assume that $(u,p)$ is a classical solution of \eqref{eq:instat1}--\eqref{eq:instat3}, such that the energies are finite. 
As a first step, we will show now that for appropriate choice of $\eps$, the two energies $\E^1$ and $\E^1_\eps$ are equivalent.
\begin{lemma} \label{lem:a3}
Let $|\eps| \le \frac{a_0}{ 4 +2 a_1}$. Then 
\begin{align*} 
\frac{1}{2} \E^1(t) \le \E^1_\eps(t) \le \frac{3}{2} \E^1(t). 
\end{align*}
\end{lemma}
\begin{proof}
We only have to estimate the additional term in the modified energy.
By the Cauchy-Schwarz inequality and the estimate of Lemma~\ref{lem:a2}, we get
\begin{align*}
(\dt u(t),u)  
&\le \|\dt u(t)\| \|u(t)\| 
\le \frac{1}{a_0} \|\dt u(t)\|^2 + \frac{a_1}{a_0} \|\dt u(t)\| \|\dt p(t)\|.
\end{align*}
Using Young's inequality to bound the last term yields
\begin{align*}
|(\dt u(t),u(t))| 
\le \frac{2+a_1}{2a_0} \|\dt u(t)\|^2 + \frac{a_1}{2a_0} \|\dt p(t)\|^2 
\le \frac{2+a_1}{2a_0} \big(\|\dt u(t)\|^2 + \|\dt p(t)\|^2 \big).
\end{align*}
The bound on $\eps$ and the definition of $\E^1(t)$ 
further yields  $|\eps (\dt u(t), u(t))| \le \frac{1}{2} \E^1(t)$,
from which the assertion of the lemma follows via the triangle inequality.
\end{proof}

We can now establish the exponential decay for the modified energy.
\begin{lemma} \label{lem:a4}
Let $0 \le \eps \le \frac{2a_0^3}{8 a_0^2+ 4 a_0^2 a_1 + 2 a_0 a_1  +a_1^4}$. 
Then 
\begin{align*}
\E^1_\eps(t)  \le e^{-2 \eps/3 (t-s)} \E^1_\eps(s).
\end{align*}
\end{lemma}
\begin{proof}
To avoid technicalities, let us assume that the solution is sufficiently smooth
first, such that all manipulations are well-defined.
By the definition of the modified energy and the energy identity given in Lemma~\ref{lem:ee1}, we have
\begin{align*}
\frac{d}{dt} \E^1_\eps(t) 
&= \frac{d}{dt} \E^1(t) + \eps \frac{d}{dt}  (\dt u(t),u(t)) \\
&\le - a_0 \|\dt u(t)\|^2 + \eps \frac{d}{dt}  (\dt u(t),u(t)).
\end{align*}
The last term can be expanded as
\begin{align} \label{eq:split}
\eps \frac{d}{dt}  (\dt u(t),u(t))
&= \eps \|\dt u(t)\|^2 + \eps (\dtt u(t), u(t)). 
\end{align}
Using the fact that $(u,p)$ as well as $(\dt u,\dt p)$ solve the variational principle, we can estimate the last term by
\begin{align*}
(\dtt u(t), u(t))
&= (\dt p(t), \dx u(t)) - (a \dt u(t), u(t))\\
&= -(\dt p(t), \dt p(t)) - (a \dt u(t), u(t))\\ 
&\le -\|\dt p(t)\|^2 + a_1 \|\dt u(t)\| \|u(t)\|.
\end{align*}
Using Lemma~\ref{lem:a2} to bound $\|u(t)\|$ and Young's inequality, we further get
\begin{align*}
(\dtt u(t), u(t)) 
&\le -\|\dt p(t)\|^2 + \frac{a_1}{a_0} \|\dt u(t)\|^2 + \frac{a_1^2}{a_0} \|\dt u(t)\| \|\dt p(t)\| \\
&\le -\frac{1}{2} \|\dt p(t)\|^2 + \big( \frac{a_1}{a_0} + \frac{a_1^4}{2a_0^2} \big) \|\dt u(t)\|^2.
\end{align*}
Inserting this estimate in \eqref{eq:split} then yields
\begin{align*}
\frac{d}{dt} \E^1_\eps(t) 
\le -\big(a_0 - \eps (1+\frac{a_1}{a_0} + \frac{a_1^4}{2 a_0^2})\big) \|\dt u(t)\|^2  
- \frac{\eps}{2} \|\dt p(t)\|^2.
\end{align*}
The two factors are balanced by the choice 
$\eps = \frac{2 a_0^3}{3 a_0^2 + 2 a_0 a_1 + a_1^4}$. In order to satisfy 
also the condition of the Lemma~\ref{lem:a3}, we enlarge the denominator 
by $5a_0^2 + 4 a_0^2a_1$, which yields the expression for $\eps$ stated in the lemma. 
In summary, we thus obtain
\begin{align*}
\frac{d}{dt} \E^1_\eps(t) \le -\eps \E^1(t) \le -\eps \frac{2}{3} \E^1_\eps(t).
\end{align*}
The result for smooth solutions now follows by integration. The general case is obtained 
by smooth approximation and continuity similar as in the proof of Lemma~\ref{lem:ee1}.
\end{proof}

\noindent
Combination of the previous estimates yields the assertion of Theorem~\ref{thm:ee2} for $k=1$.

%

\subsection{Proof of Theorem~\ref{thm:ee2} for $k=0$ and $k \ge 2$} \label{sec:a3}

We will first show how the estimate for $k=0$ can be deduced from that for $k=1$.
Let $u_0,p_0 \in L^2(0,1)$ be given and consider the following stationary problem
\begin{align*}
\dx \bar p + a \bar u &= u_0, \qquad \text{in } (0,1),  \\
\dx \bar u &= p_0, \qquad \text{in } (0,1),             
\end{align*}
with boundary condition $\bar p=0$ on $\{0,1\}$. 
Using Lemma~\ref{lem:stat}, we readily obtain 
\begin{lemma}
Let $0 < a_0 \le a(x) \le a_1$. 
Then there exists a unique strong solution $(\bar u,\bar p) \in H^1(0,1) \times H_0^1(0,1)$
and $\|\bar u\|_{H^1(0,1)} + \|\bar p\|_{H^1(0,1)} \le C \big( \|u_0\| + \|p_0\| \big)$.
\end{lemma}
Let us now define 
\begin{align*}
U(t) =  \int_0^t u(s) ds - \bar u 
\qquad \text{and} \qquad 
P(t) = \int_0^t p(s) ds - \bar p.
\end{align*}
Then $(U,P)$ is the classical solution of the damped wave system 
\eqref{eq:instat1}--\eqref{eq:instat3} with initial values $U(0)=-\bar u$ and $P(0)=-\bar p$. 
Applying Theorem~\ref{thm:ee2} for $k=1$ to $(U,P)$, we obtain
\begin{align*}
&\|u(t)\|^2 + \|p(t)\|^2 
=  \|\dt U(t)\|^2 + \|\dt P(t)\|^2 \\
& \qquad \le C e^{-\alpha (t-s) } \big(\|\dt U(s)\|^2 + \|\dt P(s)\|^2 \big)
= C e^{-\alpha (t-s)} \big( \|u(s)\|^2 + \|p(s)\|^2 \big),
\end{align*}
which yields the assertion of Theorem~\ref{thm:ee2} for $k=0$.
The result for $k \ge 2$ then follows by applying the estimate for 
$k=0$ to $(\dtk u, \dtk p)$. 

\subsection{Proof of the Theorem~\ref{thm:ee2hn} for $k=1$} 
We now turn to the fully discrete schemes. 
To establish the decay estimate for the energy $E_h^{1,n}=\frac{1}{2}(\|\bar\partial_\tau u_h^n\|^2+\|\bar\partial_\tau p_h^n\|^2)$, 
let us define the modified energy
\begin{align*}
 E_{h,\eps}^{1,n}= E_h^{1,n}+\eps(\bar\partial_\tau u_h^n,u_h^{n,\theta}).
\end{align*}
As before, the two energies $ E_h^{1,n}$ and $ E_{h,\eps}^{1,n}$ are equivalent for approriate choice of $\eps$. 
\begin{lemma}
\label{lem:a6}
Let $|\eps|<\frac{a_0}{4+2a_1}$. Then
\begin{align*}
\frac{1}{2} E_h^{1,n}\le  E_{h,\eps}^{1,n}\le \frac{3}{2} E_h^{1,n}.
\end{align*}
\end{lemma}
\noindent 
The proof of this assertion follows almost verbatim as that of Lemma \ref{lem:a3}.
With similar arguments as on the continuous level, we can then also establish the 
exponential decay estimate for the modified energy $ E_{h,\eps}^{1,n}$. 
\begin{lemma} \label{lem:a7}
Let $1/2<\theta\le 1$, $0<\eps\le \eps_0=\frac{2a_0^3}{8a_0^2+4a_0^2a_1+3a_0a_1+4a_1^4}$, 
and $0 < \tau \le \tau_0$ with 
\begin{align*} 
\tau_0=\frac{\theta-\frac{1}{2}}{\eps_0(\frac{5}{4}\theta^2+\frac{a_1}{2a_0}\theta^2+\frac{(1-\theta)^2}{4}+\frac{\theta(1-\theta)}{2})}.
\end{align*}
Then there holds
\begin{align*}
E_{h,\eps}^{1,n} 
\le e^{-\eps (n-m)\tau/3} E_{h,\eps}^{1,m} \qquad \text{for all} \qquad m \le n.
\end{align*}
\end{lemma}
\noindent
Note that the maximal step size  $\tau_0$ only depends on $a_0$, $a_1$, and the choice of $\theta$. 
The condition $\theta > 1/2$ is required here to make $\tau_0$ positive.  

\begin{proof}
 Following the arguments of the proof of Lemma~\ref{lem:a4}, we start with
 \begin{align*} 
\bar\partial_\tau E_{h,\eps}^{1,n} 
&= \bar\partial_\tau E_{h}^{1,n} + \eps \bar\partial_\tau (\bar\partial_\tau u_h^n, u_h^{n,\theta}) \\
&\le - a_0 \|\bar\partial_\tau u_h^{n,\theta}\|^2 - (\theta-\tfrac{1}{2}) \tau \big(\|\bar\partial_{\tau\tau} u_h^{n}\|^2 + \|\bar\partial_{\tau\tau} p_h^{n}\|^2 \big) + \eps \bar\partial_\tau (\bar\partial_\tau u_h^n, u_h^{n,\theta}). \notag
\end{align*}
The last term can be expanded as
\begin{align*}
 \bar\partial_\tau (\bar\partial_\tau u_h^n, u_h^{n,\theta}) = (\bar\partial_\tau u_h^n, \bar\partial_\tau u_h^{n,\theta}) + (\bar\partial_{\tau\tau} u_h^n, u_h^{n-1,\theta}).
\end{align*}
By using the identity $\bar\partial_\tau u_h^n=\bar\partial_\tau u_h^{n,\theta}+(1-\theta)\tau\bar\partial_{\tau\tau}u_h^n$, the first term of this expression yields
\begin{align*}
(\bar\partial_\tau u_h^n, \bar\partial_\tau u_h^{n,\theta})
\le 2\|\bar\partial_\tau u_h^{n,\theta}\|^2 
  + \frac{(1-\theta)^2\tau^2}{4}\|\bar\partial_{\tau\tau}u_h^n\|^2.
\end{align*}
To estimate the second term, 
we use the fact that besides  $(u_h^n,p_h^n)$ also $(\bar\partial_\tau u_h^n,\bar\partial_\tau p_h^n)$ satisfies equation \eqref{prob:theta1} and \eqref{prob:theta2}. This implies
\begin{align*}
 (\bar\partial_{\tau\tau} u_h^n, u_h^{n-1,\theta}) &=(\bar\partial_\tau p_h^{n,\theta},\partial_x u_h^{n-1,\theta}) -(a\bar\partial_\tau u_h^{n,\theta},u_h^{n-1,\theta})\\
 &=-(\bar\partial_\tau p_h^{n,\theta},\bar\partial_\tau p_h^{n-1})-(a\bar\partial_\tau u_h^{n,\theta},u_h^{n-1,\theta}).
\end{align*}
Using that $\bar\partial_\tau p_h^{n-1} = \bar\partial_\tau p_h^{n,\theta} - \theta \tau \bar\partial
_{\tau\tau} p_h^n$, we see that
\begin{align*}
-(\bar\partial_\tau p_h^{n,\theta},\bar\partial_\tau p_h^{n-1}) 
= -\|\bar\partial_\tau p_h^{n,\theta}\|^2 + \theta\tau (\bar\partial_\tau p_h^{n,\theta}, \bar\partial_{\tau\tau} p_h^n)
\le -\frac{3}{4} \|\bar\partial_\tau p_h^{n,\theta}\|^2 + \theta^2\tau^2 \|\bar\partial_{\tau\tau} p_h^n\|^2.
\end{align*}
A discrete version of Lemma~\ref{lem:a2} allows us to bound 
\begin{align*}
\|u_h^{n-1,\theta}\| \le \frac{1}{a_0} \|\bar\partial_\tau u_h^{n-1}\|  + \frac{a_1}{a_0} \|\bar\partial_\tau p_h^{n-1}\|.
\end{align*}
The remaining term in the above estimate can then be treated by
\begin{align*}
-(a\bar\partial_\tau u_h^{n,\theta},u_h^{n-1,\theta})
&\le \|\bar\partial_\tau u_h^{n,\theta}\| \big( \frac{a_1}{a_0}\|\bar\partial_\tau u_h^{n-1}\| 
 +\frac{a_1^2}{a_0}\|\bar\partial_\tau p_h^{n-1}\|\big)\\
&\le \|\bar\partial_\tau u_h^{n,\theta} \| \big( \frac{a_1}{a_0} \|\bar \partial_\tau u_h^{n,\theta}\| + \theta \tau \frac{a_1}{a_0} \|\bar\partial_{\tau\tau} u_h^n\| + \frac{a_1^2}{a_0} \|\bar\partial_\tau p_h^{n,\theta}\| + \theta \tau \frac{a_1^2}{a_0}\|\bar\partial_{\tau\tau} p_h^n\|\big), 
\end{align*}
where for the last step, we used the same expansion of $p_h^{n-1}$ as above and a similar formula for $u_h^{n-1}$. 
Via Youngs inequalities and basic manipulations, we then arrive at
\begin{align*}
-(a\bar\partial_\tau u_h^{n,\theta},u_h^{n-1,\theta})
&\le 
(\frac{3a_0a_1+4a_1^4}{2a_0^2})\|\bar\partial_\tau u_h^{n,\theta}\|^2 
+\frac{1}{4}\|\bar\partial_\tau p_h^{n,\theta}\|^2 
\\  &\qquad \qquad 
+\frac{1}{4}\theta^2\tau^2\|\bar\partial_{\tau\tau}p_h^n\|^2
+\frac{a_1}{2a_0}\theta^2\tau^2\|\bar\partial_{\tau\tau}u_h^n\|^2.
\end{align*}
In summary, we thus arrive at 
\begin{align*}
(\bar\partial_{\tau\tau} u_h^n, u_h^{n-1,\theta}) 
 &\le (\frac{3a_0a_1+4a_1^4}{2a_0^2})\|\bar\partial_\tau u_h^{n,\theta}\|^2 -\frac{1}{2}\|\bar\partial_\tau p_h^{n,\theta}\|^2 
\\ &\qquad \qquad 
+\frac{5}{4}\theta^2\tau^2\|\bar\partial_{\tau\tau}p_h^n\|^2+\frac{a_1}{2a_0}\theta^2\tau^2\|\bar\partial_{\tau\tau}u_h^n\|^2.
\end{align*}
Putting all estimates together, we finally obtain
\begin{align*}
 \bar\partial_\tau E_{h,\eps}^{1,n}&
\le
-\big(a_0-\eps \frac{4a_0^2 + 3a_0a_1+4a_1^4}{2a_0^2}\big)\|\bar\partial_\tau u_h^{n,\theta}\|^2  
-\frac{\eps}{2}\|\bar\partial_\tau p_h^{n,\theta}\|^2\\
 &\,\quad
-\big(\theta-\frac{1}{2}-\eps\tau \frac{2 a_1 \theta^2 + a_0 (1-\theta)^2}{4a_0}\big)\tau\|\bar\partial_{\tau\tau}u_h^n\|^2
-\big(\theta-\frac{1}{2}-\eps\tau\frac{5\theta^2}{4}\big)\tau\|\bar\partial_{\tau\tau}p_h^n\|^2.
\end{align*}
By the particular choice of $\tau_0$, we may estimate the terms in the second line 
from above by $-\frac{\eps}{2} \theta (1-\theta) \tau^2 \big( \|\bar\partial_{\tau\tau}u_h^n\|^2 + \|\bar\partial_{\tau\tau}p_h^n\|^2\big)$.
The two factors in the first line are balanced by the choice $\eps=\frac{2a_0^3}{5a_0^2+3a_0a_1+4a_1^4}$. In order to satisfy also the condition of Lemma \ref{lem:a6}, we enlarge the 
denominator by $3a_0^2+4a_0^2a_1$, and obtain
\begin{align*}
\bar\partial_\tau E_{h,\eps}^{1,n}
&\le 
-\eps_0 \big\{ \frac{1}{2} \big( \|\bar\partial_\tau u_h^{n,\theta}\|^2  + \|\bar\partial_\tau p_h^{n,\theta}\|^2\big) 
+ \theta (1-\theta) \frac{\tau^2}{2} \big( \|\bar\partial_{\tau\tau}u_h^n\|^2+\|\bar\partial_{\tau\tau}p_h^n\|^2\big) \big\} \\
&= -\eps_0 \big( \theta E_h^{1,n} + (1-\theta) E_h^{1,n-1} \big)
 \le  -\eps \big( \theta E_h^{1,n} + (1-\theta) E_h^{1,n-1} \big)
\end{align*}
for $\eps \le \eps_0$.
Because of the equivalence of the energies stated in Lemma~\ref{lem:a6}, this leads to
\begin{align*}
E_{h,\eps}^{1,n} \le \frac{1-\frac{2}{3} \eps (1-\theta) \tau}{1+\frac{2}{3}\eps \theta \tau} E_{h,\eps}^{1,n-1} \le(1-\frac{\eps\tau}{3})E_{h,\eps}^{1,n-1}
\le e^{-\eps \tau/3} E_{h,\eps}^{1,n-1},
\end{align*}
where we used that $ \tfrac{2}{3}\eps \theta \tau \le \tfrac{2}{3}\eps_0 \theta \tau_0 \le 1$ in the second step, which follows from the definition of $\tau_0$. 
The assertion of the Lemma now follows by induction.
\end{proof}

\begin{remark}
Let us emphasize that the assertion of  Lemma~\ref{lem:a7} holds true also for the choice 
$\theta=\frac{1}{2} + \lambda \tau$ with $\lambda$ sufficiently large, but independent of $\tau$.
\end{remark}

Using the equivalence of the discrete energies stated in Lemma~\ref{lem:a6}, we now readily obtain 
the proof of Theorem~\ref{thm:ee2hn} for the case $k=1$.

\subsection{Proof of Theorem~\ref{thm:ee2hn} for $k=0$ and $k \ge 2$}.

The assertions for $k = 0$ and $k \ge 2$ follow from the one for $k=1$ with the same arguments as on the continuous level.



\end{document}